\documentclass[reqno, a4paper, 10pt]{amsart}
\usepackage{amssymb}
\usepackage{mathrsfs}
\usepackage{amssymb, url, color, graphicx, amscd, mathrsfs}
\usepackage[colorlinks=true, bookmarks=true, pdfstartview=FitH, pagebackref=true, linktocpage=true, linkcolor = magenta, citecolor = blue]{hyperref}
\usepackage{graphicx}
\usepackage{times}

\newtheorem{theorem}{Theorem}[section]
\newtheorem{lemma}[theorem]{Lemma}
\newtheorem{corollary}[theorem]{Corollary}

\theoremstyle{definition}
\newtheorem{definition}[theorem]{Definition}

\theoremstyle{remark}
\newtheorem{remark}[theorem]{Remark}

\numberwithin{equation}{section}

\DeclareMathOperator{\Hess}{Hess}
\DeclareMathOperator{\Ric}{Ric}

\DeclareMathOperator{\dist}{dist}

\allowdisplaybreaks

\begin{document}

\title[Gradient estimates for $f$-heat equations]{Gradient estimates for some $f$-heat equations driven by Lichnerowicz's equation on complete smooth metric measure spaces}

\def\cfac#1{\ifmmode\setbox7\hbox{$\accent"5E#1$}\else\setbox7\hbox{\accent"5E#1}\penalty 10000\relax\fi\raise 1\ht7\hbox{\lower1.0ex\hbox to 1\wd7{\hss\accent"13\hss}}\penalty 10000\hskip-1\wd7\penalty 10000\box7 }

\author[N.T. Dung]{Nguyen Thac Dung}
\address[N.T. Dung]{Department of mathematics, College of Science\\ Vi\^{e}t Nam National University, Ha N\^{o}i, Vi\^{e}t Nam}
\email{\href{mailto: N.T. Dung <dungmath@gmail.com>}{dungmath@gmail.com}}

\author[N.N. Khanh]{Nguyen Ngoc Khanh}
\address[N.N. Khanh]{Department of mathematics, College of Science\\ Vi\^{e}t Nam National University, Ha N\^{o}i, Vi\^{e}t Nam}
\email{\href{mailto: N.N. Khanh <khanh.mimhus@gmail.com>}{khanh.mimhus@gmail.com}}

\author[Q.A. Ng\^o]{Qu\cfac oc Anh Ng\^o}
\address[Q.A. Ng\^{o}]{Department of mathematics, College of Science\\ Vi\^{e}t Nam National University, Ha N\^{o}i, Vi\^{e}t Nam}
\email{\href{mailto: Q. A. Ngo <bookworm\_vn@yahoo.com>}{bookworm\_vn@yahoo.com}}

\begin{abstract}
Given a complete, smooth metric measure space $(M,g,e^{-f}dv)$ with the Bakry--\'{E}mery Ricci curvature bounded from below, various gradient estimates for solutions of the following general $f$-heat equations 
$$
u_t=\Delta_f u+au\log u+bu +Au^p+Bu^{-q}
$$
and
\[
u_t=\Delta_f u+Ae^{pu}+Be^{-pu}+D
\]
are studied. As by-product, we obtain some Liouville-type theorems and Harnack-type inequalities for positive solutions of several nonlinear equations including the Schr\"{o}dinger equation, the Yamabe equation, and Lichnerowicz-type equations as special cases. 
\end{abstract}



\maketitle


\section{Introduction}

The motivation of the present work traces back to the seminal paper \cite{LY} by Li and Yau thirty years ago. In that elegant paper, the authors introduced a global version of a Harnack-type inequality for positive solutions of the following parabolic equation 
\begin{equation}\label{eqSchrodinger}
\Big(\Delta-q(x, t)-\frac{\partial}{\partial t}\Big)u(x,t)=0
\end{equation}
on complete Riemannian manifolds, where the potential $q(x, t)$ is assumed to be $C^2$ in the first variable and $C^1$ in the second variable. The main ingredient in their proof are gradient estimates. Such a Harnack-type inequality is of importance because it allows the authors to obtain an upper estimate for the heat kernel of \eqref{eqSchrodinger} if the underlying manifolds have non-negative Ricci curvature; in addition, such an upper bound is sharp. Interestingly, the Harnack inequality they obtained also exhibits a physical phenomenon in that the temperature at a given point in spacetime is controlled from the above by the temperature at a later time. 

In 1993, Hamilton \cite{H} proved a different version of gradient estimates for heat equations
\begin{equation}\label{eqHeatEquation}
\frac{\partial }{\partial t} u(x,t)=\Delta u(x,t)
\end{equation}
on compact Riemannian manifolds. Moreover, such a gradient estimate allows the author to bound $\Delta u$ from above. It is now known that the gradient estimates obtained by Hamilton are useful for proving monotonicity formulas; see \cite{GH}. Owning certain physical interpretation, Hamilton's gradient estimates also show that if the temperature is bounded, then one can compare the temperature of two different points at the same time. 

Inspired by the work of Hamilton for the heat equation \eqref{eqHeatEquation}, Souplet and Zhang \cite{SZ} proposed different gradient estimates and hence obtaining new Liouville-type theorems for the heat equations on non-compact manifolds. Their result seems surprising because it enables the comparison of temperature distribution instantaneously, without any lag in time, even for non-compact manifolds, regardless of the boundary behavior; see \cite{SZ} for detailed discussion.

Taking the understanding of the Ricci flow introduced by Hamilton \cite{Hal} into account, Ma \cite{Lima} investigated the following equation
\begin{equation}\label{eqMaEquation}
\Delta u+au\log u+bu=0
\end{equation}
on complete non-compact Riemannian manifolds where $a$ and $b$ are constant with $a<0$. His finding for local gradient estimates for positive solutions of \eqref{eqMaEquation} on complete non-compact Riemannian manifold is almost optimal if one considers Ricci solitons. Note that a Riemannian manifold $(M, g)$ is called \textit{gradient Ricci soliton} if there is a smooth function $f$ on $M$ and a constant $\lambda\in\mathbb{R}$ such that 
\[
\Ric+\Hess f=\lambda g.
\]
Given such a gradient Ricci soliton, if we set $u=e^f$, then by a simple computation, we can show that $u$ solves
$$\Delta u+2\lambda u\log u+(A_0+n\lambda)u=0$$
for some constant $A_0$; see \cite{Lima}. Raised by Ma, a natural question is that whether or not we have local gradient estimates for positive solutions to following evolution equation
\begin{equation}\label{eqMaParabolicEquation}
u_t=\Delta u+au\log u+bu.
\end{equation}

Ma's problem was generalized to the so-called smooth metric measure spaces. Recall that a \textit{complete smooth metric measure space} is a triple $(M,g,e^{-f}dv)$, where $dv$ is the volume element induced by the metric tensor $g$ and $e^{-f}dv$ is the weighted measure. On $(M, g, e^{-f}dv)$, the Bakry-\'{E}mery curvature $\Ric_f$ is related to the Ricci tensor by
\[
\Ric_f:= \Ric+ \Hess f,
\]
where $\Hess f$ is the Hessian matrix of $f$ with respect to the metric tensor $g$. It is obvious to see that gradient Ricci solitons are special smooth metric measure spaces. In \cite{Ru}, Ruan studied gradient estimates of Souplet--Zhang type for the evolution equation 
$$u_t=\Delta_fu+hu $$
where $h$ is a negative function defined on $M\times (0, +\infty)$ which is $C^1$ in the $x$-variable. Here we denote 
\[
\Delta_f \cdot = \Delta \cdot + \langle \nabla f  , \nabla \cdot \rangle.
\] 
He obtained a dimension-free elliptic-type gradient estimate which improves Souplet--Zhang's gradient estimate. Recently by considering smooth metric measure spaces with $\Ric_f$ bounded from below, Wu \cite{Wu}, Dung--Khanh \cite{DK}, and Huang--Ma \cite{HM} obtained some gradient estimates of Hamilton's and Souplet--Zhang's type for the general evolution equation
\begin{equation}\label{ricc}
u_t=\Delta_fu+au\log u+bu
\end{equation} 
where $a, b$ are $C^1$ functions in the $x$-variable defined on $M\times (0, +\infty)$. We refer the reader to \cite{DK, Wu} for further references.


The present paper is also inspired by a work due to Bidaut-V\'{e}ron and V\'{e}ron \cite{BVV}. In \cite{BVV}, for some constants $b<0$ and $p>1$ the authors considered the following Yamabe-type equation 
$$\Delta u+bu+u^p=0$$
on compact manifolds. Under some additional conditions on the Ricci tensor, the dimension constant, and the ranges of $b$, $p$, they showed that the above Yamabe-type equation has only trivial solution. When the underlying Riemannian manifold is complete, non-compact, Brandolini et al. \cite{BRS} considered the Yamabe-type equation
\begin{equation}\label{yamabe}
\Delta u+a(x)u+A(x)u^p=0,
\end{equation}
where $a(x)$ and $A(x)$ are continuous functions on $M$ and $p>1$. If $A(x)<0$ everywhere, they proved that \eqref{yamabe} has no positive bounded solution satisfying some integrable conditions. For further discussion on Yamabe's problem, we refer the reader to \cite{MRS} and the references therein.

In the literature, an analogue but more general form of Yamabe's equation is the so-called Einstein-scalar field Lichnerowicz equation. When the underlying manifold $M$ has dimension $n \geqslant 3$, the Einstein-scalar field Lichnerowicz equation takes the following form
\begin{equation}\label{eqLichnerowicz3+}
\Delta u+bu+Au^p+Bu^{-q}=0,
\end{equation}
where $a, A, B$ are smooth functions and $p=(n+2)/(n-2)$ and $q=-(3n-2)/(n-2)$ while on $2$-manifolds, the Einstein-scalar field Lichnerowicz equation is given as follows
\begin{equation}\label{eqLichnerowicz2}
\Delta u +Ae^{2u}+Be^{-2u}+D=0;
\end{equation}
see \cite{CB-2009, Ngo}. A Liouville-type result for positive solutions of a slightly generalization of \eqref{eqLichnerowicz3+} was obtained in \cite[Section 8]{Ngo}. 

Recently, Zhao \cite{Zh1, Zh2} and Song--Zhao \cite{SZh} considered the general Lichnerowicz equation
\begin{equation}\label{Lich}
\Delta_fu+bu+Au^p+Bu^{-q}=0,
\end{equation}
where $b, A, B$ are smooth functions on smooth metric measure spaces $(m, g, e^{-f}dv)$ and $p, q\geqslant 0$. They derived some gradient estimates for positive solution $u$ and proved some Harnack type inequalities.
 
Taking \eqref{ricc}, \eqref{eqLichnerowicz3+}, \eqref{eqLichnerowicz2}, and \eqref{Lich} into accounts, in this paper, we investigate the following general $f$-heat equation on $(M, g, e^{-f}dv)$
\begin{equation}\label{eqMainPDE-1}
u_t=\Delta_f u+au\log u+bu+Au^p+Bu^{-q}
\end{equation}
and
\begin{equation}\label{eqMainPDE-2}
u_t=\Delta_f u+Ae^{2u}+Be^{-2u}+D
\end{equation}
where $a$, $b$, $A$, $B$, and $D$ are constants and $p,q$ are non-negative constants. The primary aim of the paper is to obtain gradient estimates for positive, bounded solutions of \eqref{eqMainPDE-1} and \eqref{eqMainPDE-2}.

For positive, bounded solutions of \eqref{eqMainPDE-1}, there exists some constant $C>0$ such that $0 < u < C$ everywhere on $M$. Hence by the change of variable $v =u/C$, we know that $v$ also solves \eqref{eqMainPDE-1} with $b$ replaced by $b+a\log C$, $A$ replaced by $AC^{p-1}$, and $B$ replaced by  $BC^{-q-1}$. However, there holds $0<v<1$ everywhere. Similarly, for each bounded solution of \eqref{eqMainPDE-2}, there exists some constant $C>0$ such that $-C < u < C$ everywhere on $M$. Under the change of variable $v=2u+2C+1$, we easily verify that $v$ also solves \eqref{eqMainPDE-2} with $A$ replaced by $2Ae^{-2C-1}$, $B$ replaced by $2Be^{2C+1}$, $D$ replaced by $2D$, that is
\begin{equation}\label{eqMainPDE-3}
v_t=\Delta_f v+Ae^{v}+Be^{-v}+D.
\end{equation}
Furthermore, we also have $1 < v < 4C+1$ everywhere. Therefore, throughout this paper, by the boundedness of a solution $u$ of \eqref{eqMainPDE-1} we mean $0 < u < 1$ and of \eqref{eqMainPDE-3} we mean $1 < u < C$, for some positive constant $C > 0$.

We are now in position to state our results. The first set of results concerns gradient estimates for bounded solutions of \eqref{eqMainPDE-1} and \eqref{eqMainPDE-3} on a complete smooth metric measure space $(M,g,e^{-f}dv)$. In the first part, we provide several gradient estimates for bounded solutions of \eqref{eqMainPDE-1} and \eqref{eqMainPDE-3}. For bounded solutions of \eqref{eqMainPDE-1}, we obtain the following result.

\begin{theorem}\label{thmGE-I}
Let $(M,g,e^{-f}dv)$ be an $n$-dimensional complete smooth metric measure space with 
\[
\Ric_f\geqslant -(n-1)K
\] 
for some constant $K\geqslant 0$ in $B(x_0,R)$, some fixed point $x_0$ in $M$, and some fixed radius $R\geqslant 2$. Let $a$, $b$, $A$, $B$, $p$, and $q$ be constants with $A\leqslant 0$, $B\geqslant 0$, $p\geqslant 1$, and $q \geqslant 0$. Assume that $u \in (0, 1]$ is a smooth solution to the general $f$-heat equation \eqref{eqMainPDE-1} in the cylinder $Q_{R,T}=B(x_0,R)\times[t_0-T,t_0]$, where $t_0\in \mathbb{R}$ and $T>0$, that is
\[
u_t=\Delta_f u+au\log u+bu+Au^p+Bu^{-q}.
\] 
Then there exists a constant $c(n)$ such that 
\begin{equation}\label{eqGE-I}
\frac{|\nabla u|}{u}\leqslant c(n)\Big(\frac{1+\sqrt{|\alpha|R}}{R}+\frac{1}{\sqrt{t-t_0+T}}+\sqrt{K}+\sqrt{H_0}\Big)(1-\log u)
\end{equation}
in $Q_{R/2,T}$ with $t\neq t_0-T$. Here we denote $$H_0=\max\{a+b,0\}$$ and $$\alpha:=\max_{x\in B(x_0,1)}\Delta_f r(x),$$
where $r(x)$ is the distance from $x$ to $x_0$.
\end{theorem}

Next, for bounded solutions of \eqref{eqMainPDE-3}, we also obtain a similar result which can be stated as follows.

\begin{theorem}\label{thmGE-II}
Let $(M^n, g, e^{-f}dv)$ be a smooth metric measure space with $\Ric_f\geqslant -(n-1)K$. Suppose that $u$ is a positive smooth solution to the Hamiltonian constrain type equation
$$u_t=\Delta_fu+Ae^{u}+Be^{-u}+D$$ 
on $Q_{R, T}\subset M\times[0,\infty)$. If $1 < u < C$ then there exists a constant $c$ depending only on $n$ such that 
\begin{equation}\label{eqGE-II}
\frac{|\nabla u|}{\sqrt[]{u}}\leqslant c(n)\sqrt{C}\Big(\frac{1+\sqrt{|\alpha|R}+\sqrt[]{C}}{R}+\frac{1}{\sqrt{t-t_0+T}}+\sqrt{K}+\sqrt{H_2}\Big)
\end{equation}
in $Q_{R/2,T}$ with $t\neq t_0-T$, where 
\[
H_2=\max\limits_{Q_{R,T}}\Big\{Ae^u\frac{2u-1}{2u}-Be^{-u}\frac{2u+1}{2u}-\frac{D}{2u}, 0 \Big\}. 
\]
\end{theorem}
It is worth noticing that the left hand side of \eqref{eqGE-I} and \eqref{eqGE-II} are slightly different by $1/\sqrt u$. As far as we know, this is the first result of gradient estimates for parabolic equations with exponential nonlinearities.

In the next part, our result concerns gradient estimates for bounded solutions of \eqref{eqMainPDE-1} on manifolds evolving under the Ricci flow. In \cite{Bai}, Bailesteanu et al. considered a complete manifold $(M, g)$ without boundary and solutions of the forward Ricci flow. They obtained a local spacelike and spacetime gradient estimates for positive solutions of the heat equation; see also \cite{Liu} and \cite{Sun} for further discussion. In \cite{Zhang}, Zhang investigated the heat equation of the conformal Laplacian under the backward Ricci flows. He proved a qualitatively sharp, global Gaussian upper bound. Moreover, he also gave a local and global gradient estimates of the log-temperature under the backward and forward Ricci flow. In particular, Zhang considered on a complete compact manifold $(M, g)$ solutions of the backward Ricci flow
\[
\left\{
\begin{split}
u_t&=\Delta u\\
\frac{\partial}{\partial t}g(x, t)&=-2\Ric (x, t)
\end{split}
\right.
\]
with $x\in M$ and $t\in[0, T]$. He proved a local and global space-only gradient estimates for the bounded solution $u$. 

In this paper, we study a general backward Ricci flow and obtain the following gradient estimates. To illustrate our approach, we only consider bounded solutions of \eqref{ricflow} driven by the Lichnerowicz equation when $n \geqslant 3$. 

\begin{theorem}\label{thmGEUnderRicci-I}
Suppose that $M$ is a complete Riemannian without boundary and $(M, g(x, t))$ is a solution of the following backward Ricci flow
\begin{equation}\label{ricflow}
\left\{
\begin{split}
u_t&=\Delta u+au\log u+bu+Au^p+Bu^{-q}\\
\frac{\partial}{\partial t}g(x, t)&=-2\Ric (x, t)
\end{split}
\right.
\end{equation} 
with $x\in M$ and $t\in[0, T]$. Assume that $|\Ric (x, t)|\leqslant \kappa$ for some $\kappa\geqslant 0$ and for all $(x, t)\in \widetilde{Q}_{R, T}:=B(x_0, R)\times[0, T]$ for some fixed $x_0\in M$. Let $u \in (0, 1]$ be a smooth solution to the equation \eqref{ricflow} with $A\leqslant0, B\geqslant 0, p\geqslant 1, q>0$. Then there exists a constant $c(n)$ such that 
\begin{equation}\label{eqGE-UnderRicciFlow}
\frac{|\nabla u|}{u}\leqslant c(n)\Big(\frac{1+\sqrt{|\alpha|R}}{R}+\frac{1}{\sqrt[]{t}}+\sqrt{\kappa}+\sqrt{H_0}\Big)(1-\log u)
\end{equation}
in $\widetilde{Q}_{R/2,T}$, where $H_0=\max\{a+b,0\}$. 
\end{theorem}

Our final set of results concerns gradient estimates for bounded solutions of \eqref{eqMainPDE-1} and \eqref{eqMainPDE-3} on compact manifolds with boundary. 

Historically, gradient estimates for parabolic equations on compact Riemannian manifolds with convex boundary was first investigated by Li and Yau in \cite{LY}. Then, in \cite{Kasue}, Kasue studied function theoretic properties of a complete Riemannian manifold and proved some Laplacian comparison theorems on Riemannian manifolds with boundary. Later, following the Li--Yau method, Chen \cite{Chen} and Wang \cite{Wang} considered compact Riemannian manifolds with non-convex boundary satisfying an \textit{interior rolling $R$-ball condition} (see Definition \ref{def-InteriorRollingRBall}). They estimated the first Neumann eigenvalue and derived gradient estimates for global heat kernels. Recently, Wang \cite{FYWang} used the reflecting diffusion process and a conformal change of metric to obtain Li--Yau type gradient estimates and Harnack inequalities on compact manifolds with non-convex boundary. 

Inspired by the explosion of the work related to Li--Yau type gradient estimates as well as Souplet-Zhang type gradient estimates, we point out that gradient estimates of Souplet--Zhang type can be obtained on such manifolds. To the best of our knowledge, this is the first result of gradient estimates of Souplet--Zhang type for compact manifolds. That being said, for bounded solutions of \eqref{eqMainPDE-1}, we obtain the following result.

\begin{theorem}\label{thmGEWBoundary-I}
Let $M$ be a compact Riemannian manifold with boundary $\partial M$. Suppose that $\partial M$ satisfies the interior rolling $R$-ball condition. Let $K$ and $H$ be non-negative constants such that the Ricci curvature $\Ric_M$ of $M$ is bounded from below by $-K$ and the second fundamental form of $\partial M$ is bounded from below by $-H$. By choosing $R$ ``small", any positive solution $u(x,t)$ of 
\begin{equation}\label{neuman}
\left\{
\begin{split}
u_t &=\Delta u+au\log u+bu+Au^p+Bu^{-q}&\\
\frac{\partial u}{\partial \nu}\Big|_{\partial M} &=0
\end{split}
\right.
\end{equation}
where $a,b,A,B$ are constants and $p,q$ are non-negative constants, if satisfies $u \leqslant 1$, then 
\begin{equation}\label{eqGE-bdry-I}
\frac{|\nabla u|}{u}\leqslant {(1+H)}\Big(\sqrt[4]{24}\sqrt{C_1}{(1+H)}+\sqrt[4]{2}\frac{1+H}{\sqrt{t}}+\sqrt{\frac{{C_2}}{{R}}}+\frac{\sqrt[4]{C_3}}{R}\Big)(1-\log u)
\end{equation}
on $M\times(0, \infty)$. Here 
$$\begin{aligned}
C_1&=
\left\{
\begin{split}
&K+\max\{a+b,0\}+\\
&\sup\limits_{M\times[0,\infty)}\{Ap,A(p-1)u^{-1},Apu^{-1},0\}+\\
&\sup\limits_{M\times[0,\infty)}\{(-q-1)Bu^{-q-1},0\}
\end{split}
\right\},\\
C_2&=2\sqrt{24}(1+H)(n-1)H(3H+1),\\
C_3&=24 \big[ 16H^2+(1+H)H \big]^2+864(1+H)^4H^4.
\end{aligned}$$
\end{theorem}

For an explanation of the notation ``small", we refer the reader to Remark \ref{rem} for details. For bounded solutions of \eqref{eqMainPDE-3}, we obtain the following gradient estimates.

\begin{theorem}\label{thmGEWBoundary-II}
Let $M$ be a compact Riemannian manifold with boundary $\partial M$. Let $\partial M$ satisfy the interior rolling $R$-ball condition. Let $K$ and $H$ be non-negative constants such that the Ricci curvature $\Ric_M$ of $M$ is bounded from below by $-K$ and the second fundamental form elements of $\partial M$ is bounded from below by $-H$. By choosing $R$ ``small", any solution $1<u(x,t)<C$ of 
$$
\left\{
\begin{split}
u_t &=\Delta u+Ae^{u}+Be^{-u}+D&\\
\frac{\partial u}{\partial \nu}\Big|_{\partial M} &=0
\end{split}
\right.
$$
where $A,B,D,C$ are constants and $p$ is a non-negative constant satisfies the following gradient estimate
\begin{equation}\label{eqGE-bdry-II}
\frac{|\nabla u|}{\sqrt[]{u}}\leqslant 2{(1+H)}\Big(\sqrt[4]{24}\sqrt{CC_1}(1+H)+\sqrt[4]{2}(1+H)\sqrt{\frac{C}{t}}+\sqrt{\frac{C_2}{R}}+\frac{\sqrt[4]{C_3}}{R}\Big)
\end{equation}
on $M\times[0,+\infty)$, where 
\[\begin{split}
C_1&= 
\left\{
\begin{split}
&K+\sup\limits_{M\times[0,+\infty)}\Big\{Ae^u\frac{2u-1}{2u}-Be^{-u}\frac{2u+1}{2u}-\frac{D}{2u}, 0 \Big\}
\end{split}
\right\},\\
C_2&=2C\sqrt{24}(1+H)(n-1)H(3H+1),\\
C_3&=24 C^2\big[16H^2+(1+H)H \big]^2+864C^4(1+H)^4H^4.
\end{split}\]
\end{theorem}

As occurred once, it is worth reminding that the right hand side of \eqref{eqGE-bdry-I} and \eqref{eqGE-bdry-II} are also different by $1/\sqrt u$.

Finally, we would like to mention that the strategy of our proofs basically follows those in \cite{Bai, Bri, Chen, SZ, Wang, Wu}. More precisely, first we estimate the lower bound of the heat-operator acting on suitable functions in term of the heat solution. Then, we use an appropriate cut-off function and the maximum principle to obtain the desired results. These methods are, loosely speaking, well-known and used in many works; for instance, see \cite{Bai, Bri, Chen, ChenSung, DK, HM, Ru, SZ, Wang, Wu} and the references therein. However, we also would like to emphasize that to obtain gradient estimates of the Einstein-scalar field Lichnerowicz equation, our approach is slightly different from those used before. In fact, instead of introducing a new function before using the Bochner techniques as in previous works, for example, $\log u$, $u^{1/3}$, we directly apply the Bochner--Weitzenb\"{o}ck formula to a suitable function. Then we make use of the maximum principle to prove our results. 

The paper is organized as follows. In Section \ref{sec-GE-I}, we provide a proof of Theorem \ref{thmGE-I}. Several applications of Theorem \ref{thmGE-I} such that Harnack-type inequalities and Liouville-type theorems for \eqref{eqMainPDE-1} are also considered in this section. In Section \ref{sec-GE-UnderRicciFlow}, we study gradient estimates of \eqref{eqMainPDE-1} under the Ricci flow and give a proof of Theorem \ref{thmGEUnderRicci-I}. In Section \ref{sec-GE-II}, we consider gradient estimates of \eqref{eqMainPDE-3} and prove Theorem \ref{thmGE-II}. Gradient estimates for \eqref{eqMainPDE-1} on compact Riemannian manifolds with non-convex boundary are given in Section \ref{sec-GE-bdry}.



\section{Gradient estimates for (\ref{eqMainPDE-1}): Proof of Theorem \ref{thmGE-I}}
\label{sec-GE-I}

\subsection{Basic lemmas}
\label{subsec-BasicLemmas1}

By now a standard routine, to prove Theorem $\ref{thmGE-I}$ we need two basic lemmas. First, we establish the following result.

\begin{lemma}\label{lemmamain}
Let $a, b, A, B, p, q$ be real numbers. Suppose that $u$ is a positive solution to the equation 
$$u_t=\Delta u+au\log u+bu+Au^p+Bu^{-q}$$
on $Q_{R, T}:=B(x_0, R)\times[t_0-T, t_0]$, where $x_0\in M$ is a fixed point, $R>0$, and $t_0\in\mathbb{R}$. Moreover, assume that $u\leqslant 1$. Let $h=\log u$ and $w=|\nabla\log(1-h)|^2$, then the following estimate
$$
\Delta_f w-w_t\geqslant -2[(n-1)K+H_1]w+2(1-h)w^2+\frac{2h}{1-h}\left\langle\nabla w, \nabla h\right\rangle
$$
holds on $Q_{R, T}$ where 
\begin{equation}\label{eqH_1}
H_1:=
\left\{
\begin{split}
&\max\{a+b,0\}+\\
&\sup\{Ap,A(p-1)u^{-1},Apu^{-1},0\}+\\
&\sup\{(-q-1)Bu^{-q-1},0\}
\end{split}
\right\}.
\end{equation}
\end{lemma}

\begin{proof}
Upon using $h=\log u$, by a simple computation, Eq. \eqref{eqMainPDE-1} becomes 
\begin{equation} 
h_t=\Delta_fh+|\nabla h|^2+ah+b+Au^{p-1}+Bu^{-q-1}.
\end{equation}
By the Bochner--Weitzenb\"ock formula, we have
$$
\Delta_f|\nabla \psi|^2\geqslant 2|\nabla^2 \psi|^2+2\Ric_f(\nabla \psi,\nabla \psi)+2\left\langle\nabla\Delta_f \psi\nabla \psi\right\rangle.
$$
for any function $\psi$. Therefore, under the assumption $\Ric_f\geqslant -2(n-1)K$, after choosing $\psi = \log (1-h)$, we deduce that
\begin{equation}\label{d22}
\Delta_fw\geqslant -2K(n-1)w+2\left\langle\nabla\Delta_f\log(1-h),\nabla\log(1-h)\right\rangle.
\end{equation} 
Keep in mind that $(\nabla h)/(1-h)=-\nabla \log (1-h)$ and $u \nabla h = \nabla u$. To go further, we need to estimate the right hand side of \eqref{d22}. Clearly, by the rule $\Delta_f \cdot = \Delta \cdot + \langle \nabla f  , \nabla \cdot \rangle$ we easily get
\begin{align*}
\Delta_f\log(1-h)&=-\frac{\Delta_fh}{1-h}-w, 
\end{align*}
which, after using the equation satisfied by $h$, gives
\[
\Delta_f\log(1-h) = \frac{ah+b+Au^{p-1}+Bu^{-q-1}}{1-h}+(\log(1-h))_t-hw.
\]
Moreover,
\[
\left\{
\begin{split}
\nabla \Big(\frac{ah}{1-h} \Big) &= -a \nabla \log (1-h) + \frac h{1-h} \nabla \log (1-h),\\
\nabla \Big(\frac{Au^{p-1}}{1-h} \Big) &= -A (p-1) u^{p-1} \nabla \log (1-h) + \frac {Au^{p-1}} {1-h} \nabla \log (1-h),\\
\nabla \Big(\frac{Bu^{-q-1}}{1-h} \Big) &= B (-q-1) u^{-q-1} \nabla \log (1-h) + \frac {Bu^{-q-1}} {1-h} \nabla \log (1-h).
\end{split}
\right.
\]
Using this, we can estimate $\left\langle\nabla\Delta_f\log(1-h),\nabla\log(1-h)\right\rangle$ as follows
\begin{equation}\label{d23}
\begin{split}
\big\langle\nabla\Delta_f \log &  (1-h),\nabla\log(1-h)\big\rangle \\
=&-\left(a+A(p-1)u^{p-1}+B(-q-1)u^{-q-1}\right)w + \frac{w_t}2 +(1-h)w^2 \\
&-\frac{ah+b+Au^{p-1}+Bu^{-q-1}}{1-h}w+\frac{h}{1-h}\left\langle\nabla w, \nabla h\right\rangle.
\end{split}
\end{equation}
Combining \eqref{d22} and \eqref{d23}, we obtain
\begin{equation}\label{de23}
\begin{split}
\Delta_fw-w_t \geqslant &-2\left\{
\begin{split}
&a+A(p-1)u^{p-1}+B(-q-1)u^{-q-1}+\\
&\frac{ah+b+Au^{p-1}+Bu^{-q-1}}{1-h}
\end{split}
\right\}w \\
&\quad +2(1-h)w^2+\frac{2h}{1-h}\left\langle\nabla w, \nabla h\right\rangle-2(n-1)Kw \\
=&2
\left\{
\begin{split}
&-\frac{a+b}{1-h} + \Big(-p+\frac{-h}{1-h}\Big)Au^{p-1} \\
& +\Big(q+\frac{-h}{1-h}\Big)Bu^{-q-1}
\end{split}
\right\}w \\
& +2(1-h)w^2+\frac{2h}{1-h}\left\langle\nabla w, \nabla h\right\rangle-2(n-1)Kw.
\end{split}
\end{equation}
To conclude the proof, we observe that
\[
-\frac{a+b}{1-h}\geqslant -\max\{a+b,0\},
\]
that
\[
\begin{split}
\Big(-p+\frac{-h}{1-h}\Big)Au^{p-1}\geqslant &
\begin{cases}
-Ap, &\text{ if } A\geqslant 0 \text{ and } p-1\geqslant 0,\\
0,& \text{ if }A<0 \text{ and }p-1\geqslant 0,\\
-Apu^{-1}, &\text{ if }A\geqslant 0 \text{ and } p-1<0,\\
-A(p-1)u^{-1}, &\text{ if }A<0 \text{ and }p-1<0,
\end{cases}\\
=&-\sup\{Ap, A(p-1)u^{-1}, Apu^{-1},0\},
\end{split}
\]
and that
\[
\Big(q+\frac{-h}{1-h}\Big)Bu^{-q-1}\geqslant -\sup\{B(-q-1)u^{-q-1},0\}.
\]
Putting these facts and \eqref{de23} together, we obtain 
$$
\Delta_f w-w_t\geqslant -2[(n-1)K+H_1]w+2(1-h)w^2+\frac{2h}{1-h}\left\langle\nabla w, \nabla h\right\rangle,
$$
which is the desired estimate.
\end{proof}

The next ingredient in the proof of Theorem \ref{thmGE-I} is the maximum principle. To apply the maximum principle, the following cut-off function will be used in our arguments.

\begin{lemma}[see \cite{SZ, Wu}]
\label{cutoff}
Fix $t_0\in\mathbb{R}$ and $T>0$. For any give $\tau\in (t_0-T,t_0]$, there exists a smooth function $\bar\psi:[0,+\infty)\times[t_0-T,t_0]\to\mathbb{R}$ satisfying following properties
\begin{enumerate}
\item [(i)] $0\leqslant \bar{\psi}(r,t)\leqslant 1$ 
in $[0,R]\times[t_0-T,t_0]$ and $\bar\psi$ is supported in a subset of $[0,R]\times[t_0-T,t_0]$;
\item [(ii)] $\bar{\psi}(r,t)=1$ and $\bar{\psi}_r(r,t)=0$
in $[0,R/2]\times[\tau,t_0]$ and $[0,R/2]\times[t_0-T,t_0]$, respectively;
\item [(iii)] $|\bar{\psi}_t|\leqslant C (\tau-t_0+T)^{-1}\bar\psi^{1/2} $
in $[0,+\infty)\times[t_0-T,t_0]$ for some $C>0$ and $\bar{\psi}(r,t_0-T)=0$ for all $r\in [0,+\infty)$;
\item [(iv)] $- C_\epsilon\bar{\psi}^{\epsilon} /R \leqslant \bar{\psi}_r\leqslant 0$ and $|\bar{\psi}_{rr}|\leqslant C_\epsilon\bar{\psi}^\epsilon /R^2$
in $[0,+\infty)\times[t_0-T,t_0]$ for every $\epsilon \in (0,1]$ with some constant $C_\epsilon$ depending on $\epsilon$.
\end{enumerate}
\end{lemma}

\subsection{Proof of Theorem \ref{thmGE-I}}

Recall that $u$ solves the following equation
$$u_t=\Delta_fu+au\log u+bu+Au^p+Bu^{-q}.$$
Now, we prove Theorem \ref{thmGE-I}. By Lemma \ref{lemmamain}, we have the estimate
\begin{equation}\label{dung1}
\Delta_f w-w_t\geqslant -2[(n-1)K+H_0]w+2(1-h)w^2+\frac{2h}{1-h}\left\langle\nabla w, \nabla h\right\rangle,
\end{equation}
where we denote $h=\log u$ and $w=|\nabla\log(1-h)|^2$. Notice that under the circumstance $A\leqslant 0$, $B\geqslant 0$, $p\geqslant 1$, and $q\geqslant 0$, we know that
\[
H_1=H_0=\max\{0, a+b\}
\]
where the constant $H_1$ is already given in \eqref{eqH_1}. Furthermore, the desired gradient estimate $|\nabla u|/u \leqslant C(n,\alpha,R, t, t_0, K, H_0) (1-\log u)$ is equivalent to the estimate $\sqrt{w} \leqslant C(n,\alpha,R, t, t_0, K, H_0)$. Hence, to realize the theorem, it suffices to bound $w^2$ appropriately from above.

With each fixed time $\tau\in (t_0-T,t_0]$, we choose a cut-off function $\bar\psi(r,t)$ satisfying all conditions in Lemma \ref{cutoff}. To conclude the theorem, we will prove that the inequality \eqref{eqGE-I} holds at every point $(x,t)$ in $Q_{R/2,T}$. To this purpose, we first transform the cut-off function $\overline\psi$ to a new cut-off function attached with $M$. Indeed, let us define the function $\psi:M\times[t_0-T,t_0]\to \mathbb{R}$ given by
$$
\psi (x,t) =\bar\psi(d(x,x_0),t)
$$
where $x_0\in M$ is a fixed point given in the statement of the theorem. Let $(x_1,t_1)$ be a maximum point of $\psi w$ in the close set 
$$\{(x, t)\in M\times[t_0-T, \tau]: d(x, x_0)\leqslant R\}.$$
We may assume that $(\psi w)(x_1, t_1)>0$; otherwise, it follows from $(\psi w)(x_1, t_1)\leqslant 0$ that $(\psi w)(x, \tau)\leqslant 0$ for all $x\in M$ such that $d(x, x_0)\leqslant R$. However, by the definition of $\psi$, we have $\psi(x, \tau)\equiv 1$ for all $x\in M$ satisfying $d(x, x_0)\leqslant R/2$. This implies that $w(x,\tau)\leqslant 0$ when $d(x, x_0)\leqslant R/2$. Since $\tau$ is arbitrary, we conclude that \eqref{eqGE-I} holds on $Q_{R/2, T}$. Note that, according to the standard argument of Calabi \cite{Calabi}, we may also assume that $(\psi w)$ is smooth at $(x_1, t_1)$. 

Obviously at $(x_1,t_1)$, we have the following facts: $\nabla(\psi w)=0$, $\Delta_f(\psi w)\leqslant 0$, and $(\psi w)_t\geqslant 0$. Hence, still being at $(x_1, t_1)$, we get
\[
\begin{split}
0 \geqslant &\Big(\Delta_f-\frac{\partial}{\partial t}\Big)(\psi w)-\Big\langle\frac{2h}{1-h}\nabla h+2\frac{\nabla\psi}{\psi},\nabla(\psi w)\Big\rangle\\
=&\psi(\Delta_fw-w_t)+w(\Delta_f\psi-\psi_t)+2\langle\nabla\psi,\nabla w\rangle \\
&-\frac{2h}{1-h}\left\langle\nabla h, \nabla w\right\rangle\psi-\frac{2h}{1-h}\left\langle\nabla h, \nabla\psi\right\rangle w-2\frac{|\nabla\psi|^2}{\psi}w-2\left\langle\nabla\psi, \nabla w\right\rangle.
\end{split}
\]
Making use of \eqref{dung1}, we further obtain
\[
\begin{split}
0\geqslant &-2[(n-1)K+H_0](\psi w)+2\psi(1-h)w^2\\
&-\frac{2h}{1-h}\left\langle \nabla\psi,\nabla h\right\rangle w +w\Delta_f\psi-w\psi_t-2\frac{|\nabla \psi|^2}{\psi}w
\end{split}
\]
at $(x_1,t_1)$. In other words, we have just proved that
\begin{equation}\label{24}
\begin{split}
2\psi(1-h)w^2\leqslant & 2[(n-1)K+H_0]\psi w+\frac{2h}{1-h}\left\langle\nabla \psi,\nabla h\right\rangle w \\
&-w\Delta_f\psi+w\psi_t+2\frac{|\nabla\psi|^2}{\psi}w
\end{split}
\end{equation}
at $(x_1,t_1)$. We have two possible cases.

\medskip\noindent\textbf{Case 1}. If $x_1\in B(x_0,R/2)$, then for each fixed $\tau \in (t_0-T, t_0]$, there holds $\psi (\cdot, \tau) \equiv 1$ everywhere on the spacelike in $B(x_0,R/2)$ by the definition of $\psi$. By \eqref{24}, we infer 
$$
2\psi w^2\leqslant 2\psi(1-h)w^2\leqslant 2[(n-1)K+H_0]\psi w+w\psi_t
$$ 
at $(x_1,t_1)$. For arbitrary $x \in B(x_0, R/2)$, we observe that
\begin{align*} 
w(x, \tau) &=(\psi w)(x, \tau)\leqslant (\psi w)(x_1, t_1)\leqslant (\psi^{1/2}w)(x_1,t_1) \\
&\leqslant [(n-1)K+H_0]\psi^{1/2}\big|_{(x_1,t_1)}+\frac{\psi_t}{2\psi^{1/2}}\Big|_{(x_1,t_1)}\\ 
&\leqslant [(n-1)K+H_0]+ C (\tau-t_0 +T)^{-1},
\end{align*}
thanks to Lemma \ref{cutoff}(ii). Since $\tau$ can be arbitrarily chosen, we complete the proof of \eqref{eqGE-I} in this case.

\medskip\noindent\textbf{Case 2}. Suppose that $x_1 \notin B(x_0,R/2)$ where $R\geqslant 2$. From now on, we use $c$ to denote a constant depending only on $n$ whose value may change from line to line. Since $\Ric_f\geqslant -(n-1)K$ and $r(x_1,x_0)\geqslant 1$ in $B(x_0,R)$, we can apply the $f$-Laplacian comparison theorem in \cite{Bri} to get
\begin{equation}
\Delta_fr(x_1)\leqslant \alpha +(n-1)K(R-1), 
\end{equation}
where $\alpha:=\max_{x\in B(x_0,1)}\Delta_fr(x)$. Using the $f$-Laplacian comparison theorem again and thanks to Lemma \ref{cutoff}, we first have
\begin{equation}\label{25}
\begin{split}
-w\Delta_f\psi=&-w\big[\psi_r\Delta_fr+\psi_{rr}|\nabla r|^2\big] \\
=&w(-\psi_r)\Delta_fr-w\psi_{rr} \\
\leqslant &-w\psi_r\big(\alpha+(n-1)K(R-1)\big)-w\psi_{rr} \\
\leqslant &w\psi^{1/2}\frac{|\psi_{rr}|}{\psi^{1/2}}+|\alpha|\psi^{1/2}w\frac{|\psi_r|}{\psi^{1/2}}+\frac{(n-1)K(R-1)|\psi_r|}{\psi^{1/2}}\psi^{1/2}w \\
\leqslant &\frac{1}{8}\psi w^2+c\Big[\Big(\frac{|\psi_{rr}|}{\psi^{1/2}}\Big)^2+\Big(\frac{|\alpha||\psi_rw|}{\psi^{1/2}}\Big)^2+\Big(\frac{K(R-1)|\psi_r|}{\psi^{1/2}}\Big)^2\Big] \\
\leqslant &\frac{1}{8}\psi w^2+c\Big(\frac{1}{R^4}+\frac{|\alpha|^2}{R^2}+K^2\Big).
\end{split}
\end{equation}
On the other hand, by the Young inequality, we have
\begin{equation}\label{26}
\begin{split}
\frac{2h}{1-h}\left\langle\nabla\psi,\nabla h\right\rangle w 
\leqslant  &2\big[\psi(1-h)w^2\big]^{3/4}\frac{|h||\nabla\psi|}{\big[\psi(1-h)\big]^{3/4}} \\
\leqslant &\psi(1-h)w^2+c\frac{(h|\nabla\psi|)^4}{\big[\psi(1-h)\big]^3} \\
\leqslant &\psi(1-h)w^2+c\frac{h^4}{R^4(1-h)^3}. 
\end{split}
\end{equation}
By using the Cauchy--Schwarz inequality several times, we easily obtain the following estimates: first for $\psi w$
\[
2[(n-1)K+H_0]\psi w\leqslant \frac{1}{8}\psi w^2+ c(K^2+H_0^2). 
\]
then for $w\psi_t$ as follows
\[\begin{split}
w\psi_t=\psi^{1/2}w\frac{\psi_t}{\psi^{1/2}}&\leqslant \frac{1}{8}\psi w^2+c\Big(\frac{\psi_t}{\psi^{1/2}}\Big)^2 \leqslant \frac{1}{8}\psi w^2+\frac{c}{(\tau -t_0+T)^2}. 
\end{split}\]
and finally for $|\nabla\psi|^2 w/\psi$ as the following
\[\begin{split}
2\frac{|\nabla\psi|^2}{\psi}w
=2\psi^{1/2}w\frac{|\nabla\psi|^2}{\psi^{3/2}}&\leqslant \frac{1}{8}\psi w^2+c\Big(\frac{|\nabla\psi|^2}{\psi^{3/2}}\Big)^2 \leqslant \frac{1}{8}\psi w^2+\frac{c}{R^4}. 
\end{split}\]
Now, we combine \eqref{24}--\eqref{26} and all above three estimates to get
\[\begin{split}
2 \psi(1-h)w^2\leqslant & \psi(1-h)w^2+c \left\{
\begin{split}
&\frac{h^4}{R^4(1-h)^3}+K^2+H_0^2+\frac{1}{R^4} \\
&+\frac{\alpha^2}{R^2}+\frac{1}{(\tau-t_0+T)^2}
\end{split}
\right\} +\frac{1}{2}\psi w^2.
\end{split}\]
Since $1-h\geqslant 1$, this inequality implies 
\begin{align*}
\psi w^2\leqslant c\Big(\frac{h^4}{R^4(1-h)^4}+K^2+H_0^2+\frac{1}{R^4}+\frac{\alpha^2}{R^2}+\frac{1}{(\tau -t_0+T)^2}\Big).
\end{align*} 
The finally, since $\psi(\cdot, \tau) \equiv 1$ in $B(x_0,R/2)$ and $h^4/(1-h)^4\leqslant 1$, we obtain
$$
w^2(x, \tau)\leqslant (\psi w)^2(x_1,t_1)\leqslant \psi w^2(x_1,t_1)\leqslant c\Big(\frac{1}{R^4}+\frac{\alpha^2}{R^2}+\frac{1}{(\tau -t_0+T)^2}+K^2+H_0^2\Big).
$$
for all $x\in B(x_0, R/2)$. Since $\tau$ is arbitrary, this also completes the proof of \eqref{eqGE-I} in this case.

\subsection{Some applications and remarks}

This subsection is devoted for several applications of Theorem \ref{thmGE-I}. First, we obtain the following Harnack-type inequality for positive solutions of \eqref{eqMainPDE-1}

\begin{corollary}[Harnack-type inequality]\label{Harnack}
Let $(M,g,e^{-f}dv)$ be an $n$-dimensional complete smooth metric measure space with $\Ric_f\geqslant -(n-1)K$ for some constant $K\geqslant 0$ in $B(x_0,R)$, fixed $x_0$ in $M$ and $R\geqslant 2$. Assume that $u$ is a positive solution to equation \eqref{eqMainPDE-1} with $u\leqslant 1$ and $a\leqslant 0, b\leqslant 0, A\leqslant 0, B\geqslant 0, p\geqslant 1, q\geqslant 0$. Suppose that $\rho=\rho(x_1,x_2)$ is the geodesic distance between $x_1$ and $x_2$ for all $x_1,x_2$ in $M$, we have
$$
u(x_2,t)\leqslant u(x_1,t)^\beta e^{1-\beta}
$$
where $\beta=\exp (- c(n)\rho (t-t_0+T)^{-1/2}-c(n)\sqrt{K}\rho )$ and $c(n)$ is a constant depending only on$n$.
\end{corollary}

\begin{proof}
By the estimates \eqref{eqGE-I}, let $R$ tends to infinity, we have
$$
\frac{|\nabla u|}{u(1-\log u)}\leqslant c(n)\Big(\frac{1}{\sqrt{t-t_0+T}}+\sqrt{K}\Big).
$$
Let $\gamma: [0,1]\to M$ be a minimal geodesic joining $x_1$ and $x_2$ satisfying $\gamma(0)=x_2$ and $\gamma(1)=x_1$. Then 
\[\begin{split}
\log\frac{1-h(x_1,t)}{1-h(x_2,t)}=&\int_0^1\frac{d\log(1-h(\gamma(s),t))}{ds}ds \\
\leqslant &\int_0^1|\dot{\gamma}|\frac{|\nabla u|}{u(1-\log u)}ds \\
\leqslant &c(n)\Big(\frac{1}{\sqrt{t-t_0+T}}+\sqrt{K}\Big)\rho. 
\end{split}\]
Let $\beta=\exp\big(- c(n)\rho (t-t_0+T)^{-1/2} -c(n)\sqrt{K}\rho\big)$, we have
$$
\frac{1-h(x_1,t)}{1-h(x_2,t)}\leqslant \frac{1}{\beta}.
$$
Therefore, by some easy calculations, it is not hard to see that
$$u(x_2,t)\leqslant u(x_1,t)^\beta e^{1-\beta}.
$$
The proof is complete.
\end{proof}

Using Theorem \ref{thmGE-I}, we can also obtain an Liouville-type result for positive solutions of Yamabe-type equations of the form \eqref{eqYamabeType} below.

\begin{corollary}[Liouville-type result for Yamabe-type equations]\label{cor1}
Let $(M,g,e^{-f}dv)$ be an $n$-dimensional complete smooth metric measure space with $\Ric_f \geqslant 0$. Suppose that $b, A, p$ are constants satisfying $b\leqslant 0, A\leqslant 0, p\geqslant 1$. Then any smooth, positive, bounded solution $u$ of
\begin{equation} \label{eqYamabeType}
\Delta u+bu+Au^p=0
\end{equation}
must be constant.
\end{corollary}

\begin{proof}
By the assumption on $a$, $A$, and $p$, we have $H_0=0$ in Theorem \ref{thmGE-I}. Moreover, since $u$ does not depend on $t$, let $t$ tends to $\infty$, then let $R\to+\infty$ in the estimates \eqref{eqGE-I}, we obtain $|\nabla u|/u \leqslant 0$. This implies that $u$ is a constant function. The proof is complete.
\end{proof}

In the next application, we also obtain an Liouville-type result for positive solutions of Lichnerowicz-type equations of the form \eqref{eqLichnerowiczType} below.

\begin{corollary}[Liouville-type result for Lichnerowicz-type equations]\label{cor2}
Let $(M,g,e^{-f}dv)$ be an $n$-dimensional complete smooth metric measure space with $\Ric_f\geqslant 0$. Suppose that $u$ is a smooth solution to the Lichnerowicz equation 
\begin{equation} \label{eqLichnerowiczType}
\Delta_f u+bu+Au^p+Bu^{-q}=0,
\end{equation}
where $b, A, B, p, q$ are constants satisfying $b\leqslant 0, A\leqslant 0, B\geqslant 0, p\geqslant 1, q\geqslant 0$. If $u$ is bounded and positive, then $u$ is constant.
\end{corollary}

\begin{proof}
We argue as the proof of Corollary \ref{cor1}, the assumption on $b$, $A$, $B$, $p$, and $q$ implies that $H_0=0$. Moreover, since $u$ does not depend on $t$, first we send $t$ to $+\infty$, then send $R$ to $+\infty$ in the estimates \eqref{eqGE-I}, eventually, we arrive at the estimate $|\nabla u|/u\leqslant 0$. This implies that $u$ is a constant function. The proof is complete.
\end{proof}

Similarly, suppose that $p\geqslant 1, q\geqslant 0, A\leqslant 0, B\geqslant 0$ and $a+b\leqslant 0$, we obtain the following result which can be considered as a generalization of Theorem 1.2 in \cite{SZ}. We omit the detail of proof.

\begin{corollary}\label{SZ1}
Let $(M, g, e^{-f}dv)$ be a complete, non-compact smooth metric measure space with non-negative Bakry--\'{E}mery curvature. Let $u$ be a strictly positive ancient solution to the heat equation \eqref{eqMainPDE-1} in the sense that it is a solution defined in all space and negative time. We assume in addition that $u$ has the following asymptotic behavior
\[
u(x, t)=\exp (o(\rho(x)+ \sqrt{|t|}) )
\]
near infinity, where $\rho(x)$ is the distance from $x$ to a fixed point $x_0\in M$. Then $u$ is a constant
\end{corollary}


\section{Gradient estimates for (\ref{eqMainPDE-1}) under the Ricci flow: Proof of Theorem \ref{thmGEUnderRicci-I}}
\label{sec-GE-UnderRicciFlow}

In this section, we derive gradient estimates for solutions of the heat equation \eqref{eqMainPDE-1} under the backward Ricci flow. Recall the system that $u$ and $g$ solve
\begin{equation} \label{ricflow1}
\left\{
\begin{split}
u_t&=\Delta u+au\log u+bu+Au^p+Bu^{-q}\\
\frac{\partial}{\partial t}g(x,t)&=-2\Ric (x, t)
\end{split}
\right.
\end{equation} 
with $x\in M$ and $t\in[0, T]$. To prove Theorem \ref{thmGEUnderRicci-I}, we follow the procedure used in the proof of Theorem \ref{thmGE-I}. First, we start with a basic lemma in the same fashion of Lemma \ref{lemmamain}.

\begin{lemma}\label{lemmamain1}
Let $(M, g(x, t))_{t\in[0, T]}$ be a complete solution to the Ricci flow 
$$ \frac{\partial}{\partial t}g(x, t)=-2\Ric (x, t) $$
 and $u$ be smooth positive solution to the heat equation 
$$ u_t=\Delta u+au\log u+bu+Au^p+Bu^{-q}. $$
Suppose that $u\leqslant 1$ for all $(x, t)\in Q_{R, T}:=B(x_0, R)\times[0, T]$. Denote $h=\log u$ and $w=|\nabla\log(1-h)|^2$. Then on the cylinder $Q_{R, T}$, we have
\begin{equation} \label{qa1}
\Delta w-w_t\geqslant -2H_1w+2(1-h)w^2+\frac{2h}{1-h}\left\langle \nabla w, \nabla h\right\rangle,
\end{equation}
where
\[
H_1:= \left\{
\begin{split}
&\max\{a+b,0\}+\sup\{Ap, A(p-1)u^{-1}, Apu^{-1},0\}\\
&+\sup\{(-q-1)Bu^{-q-1},0\}
\end{split}
\right\}.
\]
\end{lemma}

\begin{proof}
Since $u$ is a solution to the heat equation 
$$ u_t=\Delta u+au\log u+bu+Au^p+Bu^{-q}, $$
the function $h=\log u$ satisfies 
\begin{equation} 
h_t=\Delta h+|\nabla h|^2+ah+b+Au^{p-1}+Bu^{-q-1}.
\end{equation}
The Bochner--Weitzenb\"ock formula applied to $\log(1-h)$ gives
\[
\begin{split}
\Delta |\nabla(\log(1-h))|^2
&=2\Hess^2(\log(1-h))+2\left\langle\nabla\Delta \log(1-h),\nabla\log(1-h)\right\rangle \\
&\quad+2\Ric (\nabla\log(1-h),\nabla\log(1-h)) .
\end{split}
\]
Hence,
\begin{equation}\label{k22}
\Delta w\geqslant 2\Ric (\nabla\log(1-h),\nabla\log(1-h))+2\left\langle\nabla\Delta \log(1-h),\nabla\log(1-h)\right\rangle.
\end{equation}
As routine, to estimate $\Delta w$, we first estimate $\left\langle\nabla\Delta \log(1-h),\nabla\log(1-h)\right\rangle$. Clearly,
\[
\begin{split}
\Delta\log(1-h) 
&=(1-h)w+\frac{ah+b+Au^{p-1}+Bu^{-q-1}}{1-h}+(\log(1-h))_t-w \\
&=\frac{ah+b+Au^{p-1}+Bu^{-q-1}}{1-h}+(\log(1-h))_t-hw. 
\end{split}
\]
On the other hand, by the equation $\partial_t g(x, t)=-2\Ric (x, t) $, we have
$$w_t=2\langle \nabla \log(1-h),\nabla(\log(1-h)_t)\rangle+2\Ric (\nabla\log(1-h),\nabla\log(1-h)). $$
Therefore, 
\[
\begin{split}
\big\langle\nabla  \Delta & \log(1-h),\nabla\log(1-h)\big\rangle \\
=& \Big\langle\nabla \Big(\frac{ah+b+Au^{p-1}+Bu^{-q-1}}{1-h}+(\log(1-h))_t-hw\Big),\nabla\log(1-h)\Big\rangle \\
= &\left(a+A(p-1)u^{p-1}+B(-q-1)u^{-q-1}\right)w- \frac{ah+b+Au^{p-1}+Bu^{-q-1}}{1-h}w \\
& + \frac{w_t}2 -\Ric (\nabla\log(1-h),\nabla\log(1-h))+ (1-h)w^2+\frac{ h}{1-h}\left\langle \nabla w, \nabla h\right\rangle. 
\end{split}
\]
Thus, we can further estimate \eqref{k22} as follows
\begin{equation}\label{23}
\begin{split}
\Delta w-w_t \geqslant &-2\left\{
\begin{split}
&a+A(p-1)u^{p-1}+B(-q-1)u^{-q-1}\\
&+\frac{ah+b+Au^{p-1}+Bu^{-q-1}}{1-h}
\end{split}
\right\}w \\
&+2(1-h)w^2+\frac{2h}{1-h}\langle\nabla w,\nabla h\rangle \\
=&2\left\{
\begin{split}
&-\frac{a+b}{1-h}w+ \Big(-p+\frac{-h}{1-h}\Big)Au^{p-1}\\
&+\Big(q+\frac{-h}{1-h}\Big)Bu^{-q-1}
\end{split}
\right\}w  \\
&+2(1-h)w^2 +\frac{2h}{1-h}\left\langle\nabla w, \nabla h\right\rangle . 
\end{split}
\end{equation}
As in the proof of Lemma \ref{thmGEUnderRicci-I}, we have that
\[
-\frac{a+b}{1-h}\geqslant -\max\{a+b,0\},
\]
that
\[\begin{split}
\Big(-p+\frac{-h}{1-h}\Big)Au^{p-1}\geqslant &
\begin{cases}
-Ap, &\text{ if } A\geqslant 0 \text{ and } p-1\geqslant 0,\\
0, &\text{ if }A<0 \text{ and }p-1\geqslant 0,\\
-Apu^{-1}, &\text{ if }A\geqslant 0 \text{ and } p-1<0,\\
-A(p-1)u^{-1}, &\text{ if }A<0 \text{ and }p-1<0,
\end{cases} \\
=&-\sup\{Ap,A(p-1)u^{-1},Apu^{-1},0\},
\end{split}
\]
and that
\[
\begin{split}
\Big(q+\frac{-h}{1-h}\Big)Bu^{-q-1}\geqslant -\sup\{B(-q-1)u^{-q-1},0\}. 
\end{split}
\]
Combining \eqref{23} and above three estimates, we obtain 
$$
\Delta w-w_t\geqslant -2H_1w+2(1-h)w^2+\frac{2h}{1-h}\left\langle\nabla w\nabla h\right\rangle.
$$
The proof is complete.
\end{proof}

Let us recall the following cut-off function in \cite{Bai, SZ, Zhang}.

\begin{lemma}[see \cite{Bai, SZ, Zhang}]\label{cutoff1}
Given $\tau\in[0, T]$, there exists a smooth cut-off function $\overline{\psi}(r,t)$ supported in $[0, R]\times[0, T]$ satisfying
\begin{enumerate}
\item $0\leqslant \overline{\psi}(r, t)\leqslant 1$ in $[0, R]\times[0, T]$.
\item $\overline{\psi}(r, t)=1$ in $[0, R/2]\times[\tau, T]$ and $\partial_r\overline{\psi}(r, t)=0$ in $[0, R/2]\times[0, T]$.
\item When $0<\alpha\leqslant 1$, there is a constant $C_\alpha$ such that 
$$-\frac{C_\alpha\overline{\psi}^\alpha}{R}\leqslant \frac{\partial\overline{\psi}}{\partial r}\leqslant 0;\text{ and }\left|\frac{\partial^2\overline{\psi}}{\partial r^2}\right|\leqslant \frac{C_\alpha\overline{\psi}^\alpha}{R^2}.$$
\item $\overline{\psi}(r, 0)=0$ for all $r\in[0, \infty)$ and $|\partial_t\overline{\psi}|\leqslant \bar{C} \tau^{-1} \overline{\psi}^{1/2}$ on $[0, \infty)\times[0, T]$.
\end{enumerate} 
\end{lemma}

Now we are ready to prove Theorem \ref{thmGEUnderRicci-I}. 

\begin{proof}[Proof of Theorem \ref{thmGEUnderRicci-I}]
To prove Theorem \ref{thmGEUnderRicci-I}, we follows the same procedure used previously to prove Theorem \ref{thmGE-I}; hence in view of Lemma \ref{lemmamain1}, it suffices to bound $w^2$ appropriately from above. Define the cut-off function $\psi:M\times[t_0-T,t_0]\to \mathbb{R}$ such that 
$$
\psi (x,t) =\bar\psi(d(x,x_0),t)
$$
where $x_0\in M$ is the fixed point mentioned in the statement of the theorem. Let $(x_1,t_1)$ be a maximum point for the function $\psi w$ in the close set 
$$(x, t)\in \widetilde{Q}_{R, T}:=B(x_0, R)\times[0, T].$$
We may assume that $(\psi w)(x_1, t_1)>0$; otherwise, the condition $(\psi w)(x_1, t_1)\leqslant 0$ implies that $(\psi w)(x, \tau)\leqslant 0$ for all $x\in M$ such that $d(x, x_0)\leqslant R$. However, by the definition of $\psi$, we have $\psi(x, \tau)\equiv 1$ for all $x\in M$ satisfying $d(x, x_0)\leqslant R/2$. This implies that $w(x, \tau)\leqslant 0$ whenever $d(x, x_0)\leqslant R/2$. Since $\tau$ is arbitrary, we conclude that the inequality \eqref{eqGE-UnderRicciFlow} holds on $\widetilde{Q}_{R/2, T}$ as claimed. In addition to the sign convention, by a standard argument of Calabi \cite{Calabi}, we may assume that $\psi w$ is smooth at $(x_1, t_1)$. 

As always, at $(x_1,t_1)$, we have the following facts: $\nabla(\psi w)=0$, $\Delta_f(\psi w)\leqslant 0$, and $(\psi w)_t\geqslant 0$. Performing a similar argument as in the proof of Theorem \ref{thmGE-I}, we easily obtain
\[\begin{split}
0\geqslant &\Big(\Delta_f-\frac{\partial}{\partial t}\Big)(\psi w)-\Big\langle\frac{2h}{1-h}\nabla h+2\frac{\nabla\psi w}{\psi},\nabla(\psi w)\Big\rangle \\
=&\psi(\Delta_fw-w_t)+w(\Delta_f\psi-\psi_t)+2\langle\nabla\psi,\nabla w\rangle \\
&-\frac{2h}{1-h}\left\langle\nabla h, \nabla w\right\rangle\psi-\frac{2h}{1-h}\left\langle\nabla h, \nabla\psi\right\rangle w-2\frac{|\nabla\psi|^2}{\psi}w-2\left\langle\nabla\psi, \nabla w\right\rangle
\end{split}\]
at $(x_1,t_1)$. Making use of \eqref{qa1}, we further obtain
\[\begin{split}
0 \geqslant &-2H_0(\psi w)+2\psi(1-h)w^2-\frac{2h}{1-h}\left\langle \nabla\psi,\nabla h\right\rangle w +w\Delta_f\psi-w\psi_t-2\frac{|\nabla \psi|^2}{\psi}w 
\end{split}\]
at $(x_1,t_1)$. Equivalently, the preceding inequality can be rewritten as follows
\begin{equation}\label{d24}
2\psi(1-h)w^2\leqslant  2H_0\psi w+\frac{2h}{1-h}\left\langle\nabla \psi,\nabla h\right\rangle w -w\Delta_f\psi+w\psi_t+2\frac{|\nabla\psi|^2}{\psi}w.
\end{equation}
We now consider two possible cases.

\medskip\noindent\textbf{Case 1}. If $x_1\in B(x_0,R/2)$, then by the definition of $\psi$, we know that $\psi (\cdot, \tau)$ is constant on any spacelike in $B(x_0,R/2)$. Hence, it follows from \eqref{d24} that
$$
2\psi w^2\leqslant 2\psi(1-h)w^2\leqslant 2H_0\psi w+w\psi_t
$$ 
at $(x_1, t_1)$. For arbitrary $x \in B(x_0, R/2)$, we observe that
\begin{align*} 
w(x, \tau) &= \psi w(x, \tau) \leqslant (\psi^{1/2}w)(x_1,t_1) \\
&\leqslant H_0\psi^{1/2}(x_1, t_1)+\frac{\psi_t}{2\psi^{1/2}}(x_1, t_1)\\ 
&\leqslant H_0+ C(\tau),
\end{align*}
thanks to Lemma \ref{cutoff}. From this we obtain the desired estimate.

\medskip\noindent\textbf{Case 2}. Suppose that $x_1\notin B(x_0,R/2)$ with $R \geqslant 2$. For simplicity, we denote by $c$ a generic constant whose value may change from line to line. Since $\Ric_f\geqslant -(n-1)\kappa$ and $r(x_1,x_0)\geqslant 1$ in $B(x_0,R)$, we can apply the $f$-Laplacian comparison theorem \cite{Bri} to obtain
\begin{equation}
\Delta_fr(x_1)\leqslant \alpha +(n-1)\kappa(R-1), 
\end{equation}
where $\alpha:=\max_{x\in B(x_0,1)}\Delta_fr(x)$. By similar computations as in the proof of Theorem \ref{thmGE-I}, we arrive at
\begin{equation}\label{k1}
\left\{
\begin{split}
\frac{2h}{1-h}\left\langle\nabla \psi,\nabla h\right\rangle w&\leqslant \psi(1-h)w^2+c\frac{h^4}{R^4(1-h)^3},\\
-w\Delta_f\psi&\leqslant \frac{1}{8}\psi w^2+c\Big(\frac{1}{R^4}+\frac{|\alpha|^2}{R^2}+\kappa^2\Big) ,\\
2\frac{|\nabla\psi|^2}{\psi}w&\leqslant \frac{1}{8}\psi w^2+\frac{c}{R^4}.
\end{split}
\right.
\end{equation}
A direct calculation implies
\[\begin{split}
(w\psi_t)(x_1, t_1)=&w(x_1, t_1)\frac{\partial\overline{\psi}}{\partial t}(\dist (x_1, x_0, t_1), t_1) \\
&+w(x_1, t_1)\frac{\partial\overline{\psi}}{\partial r}(\dist (x_1, x_0, t_1), t_1)\Big(\frac{\partial}{\partial t}\dist (x_1, x_0, t_1)\Big). 
\end{split}\]
Note that under the assumption $|\Ric (x, t)|\leqslant \kappa$, it was proved in \cite[Eq. (2.8)]{Bai} that 
\[
\Big|\frac{\partial}{\partial t}\dist (x_1, x_0, t)\Big|\leqslant \kappa R.
\]
Therefore,
\begin{equation}\label{k2}
\begin{split}
(w\psi_t)(x_1, t_1)&\leqslant w(x_1, t_1)\Big|\frac{\partial\overline{\psi}}{\partial t}(x_1, t_1)\Big|+\kappa Rw(x_1, t_1)\Big|\frac{\partial\overline{\psi}}{\partial r}(x_1, t_1)\Big| \\
&\leqslant \frac{1}{16}(\psi w^2)(x_1, t_1)+c\Big(\frac{1}{\tau^2}+\kappa w\psi^{1/2}\Big)(x_1, t_1) \\
&\leqslant \frac{1}{8}(\psi w^2)(x_1, t_1)+c\Big(\frac{1}{\tau^2}+\kappa^2\Big)(x_1, t_1).
\end{split}
\end{equation}
On the other hand, by the Cauchy--Schwarz inequality, we have 
\begin{equation}\label{k3}
2H_0\psi w\leqslant \frac{1}{8}\psi w^2+ cH_0^2.
\end{equation}
Combining \eqref{d24}-\eqref{k3}, we conclude that
\[
2 \psi(1-h)w^2\leqslant   \psi(1-h)w^2+c\left\{
\begin{split}
&\frac{h^4}{R^4(1-h)^3}+\kappa^2+H_0^2\\&+\frac{1}{R^4}  +\frac{\alpha^2}{R^2}+\frac{1}{\tau^2}
\end{split}
\right\}+\frac{1}{2}\psi w^2.
\]
Since $1-h\geqslant 1$, we further obtain
\begin{align*}
\psi w^2\leqslant c\Big(\frac{h^4}{R^4(1-h)^4}+\kappa^2+H_0^2+\frac{1}{R^4}+\frac{\alpha^2}{R^2}+\frac{1}{\tau^2}\Big).
\end{align*} 
Finally, since $\psi(x,\tau)=1$ when $x\in B(x_0,R/2)$ and $h^4/(1-h)^4\leqslant 1$, we obtain
$$
w^2(x,\tau)\leqslant \psi w^2(x_1,t_1)\leqslant c\Big(\frac{1}{R^4}+\frac{\alpha^2}{R^2}+\frac{1}{\tau^2}+\kappa^2+H_0^2\Big).
$$
for all $x\in B(x_0, R/2)$. From this we obtain the desired result.
\end{proof}


\section{Gradient estimates for (\ref{eqMainPDE-3}): Proof of Theorem \ref{thmGE-II}}
\label{sec-GE-II}

\subsection{A basic lemma}

Now as a routine, to prove Theorem $\ref{thmGE-II}$ we need the following basic lemma. 

\begin{lemma}\label{lemmaeqMainPDE-3}
Let $ A, B$ be real numbers. Suppose that $u$ is a  solution to the equation 
$$u_t=\Delta_f u+Ae^u+Be^{-u}+D$$
on $Q_{R, T}:=B(x_0, R)\times[t_0-T, t_0]$, where $x_0\in M$ is a fixed point, $R>0$, and $t_0\in\mathbb{R}$. Moreover, assume that $1 < u < C$. Let $w=|\nabla u^{1/2}|^2$, then the following estimate
\begin{equation}\label{Eqmain4}
\Delta_fw-w_t\geqslant -2[K(n-1)+H]w-\dfrac{2}{u^{1/2}}\langle\nabla w,\nabla u^{1/2}\rangle+2\dfrac{w^2}{u}
\end{equation}
holds on $Q_{R, T}$ where 
\[
H=\max \left\{Ae^u\dfrac{2u-1}{2u}-Be^{-u}\dfrac{2u+1}{2u}-\dfrac{D}{2u},0\right\}.
\]
\end{lemma}
\begin{proof}To prove Lemma \ref{lemmaeqMainPDE-3}, we introduce a new function 
\[
w=|\nabla u^{1/2}|^2.
\]
By direct computations, we obtain
\[\begin{aligned}
\Delta_fu^{1/2}
&=\dfrac{1}{2u^{1/2}}\Delta_fu-\dfrac{1}{4}u^{-3/2}|\nabla u|^2\\
&=\dfrac{1}{2u^{1/2}}(u_t-Ae^u-Be^{-u}-D)-\dfrac{w}{u^{1/2}}.
\end{aligned}\]
Therefore, the Bochner--Weitzenb\"ock formula implies
\begin{align*}
\Delta_f w\geq&-2K(n-1)w+2\langle\nabla\Delta_f u^{1/2},\nabla u^{1/2}\rangle\\
=&-2K(n-1)w+2\Big\langle \nabla\Big (-\dfrac{w}{u^{1/2}}+\dfrac{1}{2u^{1/2}}(u_t-Ae^u-Be^{-u}-D)\Big ),\nabla u^{1/2}\Big \rangle\\
=&-2K(n-1)w-2\dfrac{1}{u^{1/2}}\langle\nabla w,\nabla u^{1/2}\rangle+2\dfrac{w^2}{u}+w_t\\
&\quad\quad -\Big\langle \nabla\Big (\dfrac{1}{u^{1/2}}(Ae^u+Be^{-u}+D)\Big ),\nabla u^{1/2}\Big \rangle\\
=&-2K(n-1)w-2\dfrac{1}{u^{1/2}}\langle\nabla w,\nabla u^{1/2}\rangle+2\dfrac{w^2}{u}+w_t\\
&\quad\quad +\dfrac{w}{u}(Ae^u+Be^{-u}+D)-(Ae^u-Be^{-u})\Big\langle \dfrac{\nabla u}{u^{1/2}}, \nabla u^{1/2}\Big \rangle\\
\geqslant & -2[K(n-1)+H]w-\dfrac{2}{u^{1/2}}\langle\nabla w,\nabla u^{1/2}\rangle+2\dfrac{w^2}{u}+w_t.
\end{align*}
The proof is complete.
\end{proof}

Note that if we still use $w=|\nabla\log(1-h)|^2$, then a similar estimate from below for $\Delta_fw-w_t$ is also available. Such an estimate also leads us to gradient estimates for bounded solutions of \eqref{eqMainPDE-3}. However, we hardly obtain Liouville-type results for bounded solutions of \eqref{eqMainPDE-2} from these gradient estimates. This forces us to obtain suitable and new gradient estimates.

\subsection{Proof of Theorem \ref{thmGE-II}}

With each fixed time $\tau\in (t_0-T,t_0]$, we choose a cut-off function $\bar\psi(r,t)$ satisfying all conditions in Lemma \ref{cutoff}. To conclude the theorem, we will prove that the inequality \eqref{eqGE-II} holds at every point $(x,t)$ in $Q_{R/2,T}$. To this purpose, we first transform the cut-off function $\overline\psi$ to a new cut-off function attached with $M$. Indeed, let us define the function $\psi:M\times[t_0-T,t_0]\to \mathbb{R}$ given by
$$
\psi (x,t) =\bar\psi(d(x,x_0),t)
$$
where $x_0\in M$ is a fixed point given in the statement of the theorem. Let $(x_1,t_1)$ be a maximum point of $\psi w$ in the close set 
$$\{(x, t)\in M\times[t_0-T, \tau]: d(x, x_0)\leqslant R\}.$$
We may assume that $(\psi w)(x_1, t_1)>0$; otherwise, it follows from $(\psi w)(x_1, t_1)\leqslant 0$ that $(\psi w)(x, \tau)\leqslant 0$ for all $x\in M$ such that $d(x, x_0)\leqslant R$. However, by the definition of $\psi$, we have $\psi(x, \tau)\equiv 1$ for all $x\in M$ satisfying $d(x, x_0)\leqslant R/2$. This implies that $w(x,\tau)\leqslant 0$ when $d(x, x_0)\leqslant R/2$. Since $\tau$ is arbitrary, we conclude that \eqref{eqGE-II} holds on $Q_{R/2, T}$. Note that, according to the standard argument of Calabi \cite{Calabi}, we may also assume that $(\psi w)$ is smooth at $(x_1, t_1)$. 

Obviously at $(x_1,t_1)$, we have the following facts: $\nabla(\psi w)=0$, $\Delta_f(\psi w)\leqslant 0$, and $(\psi w)_t\geqslant 0$. Hence, still being at $(x_1, t_1)$, we get
\[
\begin{split}
0 \geqslant \Delta_f(\psi w)-(\psi w)_t=\psi(\Delta_fw-w_t)+w(\Delta_f\psi-\psi_t)+2\langle\nabla w,\nabla\psi\rangle
\end{split}
\]
Making use of \eqref{Eqmain4}, we further obtain
\[
\begin{split}
0\geqslant &-2[(n-1)K+H](\psi w)+2\dfrac{w^2}{u}\psi\\
&+\dfrac{2}{u^{1/2}}\big\langle \nabla\psi,\nabla u^{1/2}\big\rangle w +w\Delta_f\psi-w\psi_t-2\dfrac{|\nabla \psi|^2}{\psi}w
\end{split}
\]
at $(x_1,t_1)$. In other words, we have just proved that at $(x_1, t_1)$,
\begin{equation}\label{2d4}
\begin{split}
2P\psi w^2\leqslant & 2[(n-1)K+H]\psi w-\dfrac{2}{u^{1/2}}\big\langle\nabla \psi,\nabla u^{1/2}\big\rangle w \\
&-w\Delta_f\psi+w\psi_t+2\dfrac{|\nabla\psi|^2}{\psi}w
\end{split}
\end{equation}
where 
\[P=\min\limits_{Q_{R, T}}\dfrac{1}{u}=\frac{1}{C}.\]
We have two possible cases.

\medskip\noindent\textbf{Case 1}. If $x_1\in B(x_0,R/2)$, then for each fixed $\tau \in (t_0-T, t_0]$, there holds $\psi (\cdot, \tau) \equiv 1$ everywhere on the spacelike in $B(x_0,R/2)$ by the definition of $\psi$. By \eqref{2d4}, we yield
$$
2P\psi w^2\leqslant 2P\psi(1-h)w^2\leqslant 2[(n-1)K+H]\psi w+w\psi_t
$$ 
at $(x_1,t_1)$. For arbitrary $x \in B(x_0, R/2)$, we observe that
\begin{align*} 
Pw(x, \tau) &=P\psi w(x, \tau)\leqslant P(\psi^{1/2}w)(x_1,t_1) \\
&\leqslant [(n-1)K+H]\psi^{1/2}\big|_{(x_1,t_1)}+\dfrac{\psi_t}{2\psi^{1/2}}\Big|_{(x_1,t_1)}\\ 
&\leqslant [(n-1)K+H]+ c (\tau-t_0 +T)^{-1},
\end{align*}
thanks to Lemma \ref{cutoff}(ii). Since $\tau$ can be arbitrarily chosen, we complete the proof of \eqref{eqGE-II} in this case.

\medskip\noindent\textbf{Case 2}. Suppose that $x_1 \notin B(x_0,R/2)$ where $R\geqslant 2$. From now on, we use $c$ to denote a constant depending only on $n$ whose value may change from line to line. Since $\Ric_f\geqslant -(n-1)K$ and $r(x_1,x_0)\geqslant 1$ in $B(x_0,R)$, we can apply the $f$-Laplacian comparison theorem in \cite{Bri} to get
\begin{equation}
\Delta_fr(x_1)\leqslant \alpha +(n-1)K(R-1), 
\end{equation}
where $\alpha:=\max_{x\in B(x_0,1)}\Delta_fr(x)$. Recall that as in the proof of \eqref{25}, this $f$-Laplacian comparison theorem and Lemma \ref{cutoff} implies
\begin{equation}\label{2d5}
\begin{split}
-w\Delta_f\psi
\leqslant &\dfrac{P}{8}\psi w^2+\frac{c}{P}\Big(\dfrac{1}{R^4}+\dfrac{|\alpha|^2}{R^2}+K^2\Big).
\end{split}
\end{equation}
Now, we want to estimate the second term in the right hand side of \eqref{2d4}. By the Young inequality, we have
\begin{equation}\label{2d6}
\begin{split}
-\dfrac{2}{u^{1/2}}\big\langle\nabla\psi,\nabla u^{1/2}\big\rangle w 
\leqslant  &2\big[P\psi w^2\big]^{3/4}\dfrac{|\nabla\psi|}{u^{1/2}(P\psi)^{3/4}} \\
\leqslant &P\psi w^2+\dfrac{c}{P^3}\dfrac{|\nabla\psi|^4}{\psi^3}\\
\leqslant &P\psi w^2+\dfrac{c}{P^3R^4}
\end{split}
\end{equation}
Here we used $1\leqslant u$ in the second inequality. To estimate the rest of the right hand side of \eqref{2d4}, we use the Cauchy--Schwarz inequality several times. First for $\psi w$, we have
\[
2[(n-1)K+H]\psi w\leqslant \dfrac{P}{8}\psi w^2+ \dfrac{c}{P}(K^2+H^2). 
\]
then for $w\psi_t$, we obtain
\[\begin{split}
w\psi_t=\psi^{1/2}w\dfrac{\psi_t}{\psi^{1/2}}&\leqslant \dfrac{P}{8}\psi w^2+\dfrac{c}{P}\Big(\dfrac{\psi_t}{\psi^{1/2}}\Big)^2 \leqslant \dfrac{P}{8}\psi w^2+\dfrac{c}{P}\dfrac{1}{(\tau -t_0+T)^2}. 
\end{split}\]
and finally for $|\nabla\psi|^2 w/\psi$, we yield
\[\begin{split}
2\dfrac{|\nabla\psi|^2}{\psi}w
=2\psi^{1/2}w\dfrac{|\nabla\psi|^2}{\psi^{3/2}}&\leqslant \dfrac{P}{8}\psi w^2+\dfrac{c}{P}\Big(\dfrac{|\nabla\psi|^2}{\psi^{3/2}}\Big)^2 \leqslant \dfrac{P}{8}\psi w^2+\dfrac{c}{P}\frac{1}{R^4}. 
\end{split}\]
Combining \eqref{2d4}--\eqref{2d6} and all above three estimates, we conclude that
\[\begin{split}
2 P\psi w^2\leqslant & P\psi w^2+\dfrac{c}{P} \left\{
\begin{split}
&\dfrac{1}{R^4}+\frac{1}{P^2R^4}+K^2+H^2 \\
&+\dfrac{\alpha^2}{R^2}+\dfrac{1}{(\tau -t_0+T)^2}
\end{split}
\right\} +\dfrac{P}{2}\psi w^2.
\end{split}\]
Hence, we get 
\[
\psi w^2\leqslant \dfrac{c}{P^2}\Big(\frac{1}{R^4}+\dfrac{1}{P^2R^4}+K^2+H^2+\dfrac{\alpha^2}{R^2}+\dfrac{1}{(\tau -t_0+T)^2}\Big).
\]
Keep in mind that $\psi(\cdot, \tau) \equiv 1$ everywhere in $B(x_0,R/2)$, we infer
\[
w^2(x, \tau)\leqslant \psi w^2(x_1,t_1)\leqslant cC^2\Big(\dfrac{1}{R^4}+\dfrac{C^2}{R^4}+\frac{\alpha^2}{R^2}+\frac{1}{(\tau -t_0+T)^2}+K^2+H^2\Big).
\]
for all $x\in B(x_0, R/2)$. Since $\tau$ is arbitrary, this also completes the proof of \eqref{eqGE-II} in this case.

\subsection{Applications and remarks} 

In this subsection, we make use of Theorem \ref{thmGE-II} to obtain Liouville-type results for the Einstein-scalar field Lichnerowicz equation \eqref{eqMainPDE-2}.

\begin{corollary}[Liouville-type result]\label{cor42}
Let $(M,g,e^{-f}dv)$ be an $n$-dimensional complete smooth metric measure space with $\Ric_f\geqslant 0$. Suppose that $u$ is a smooth solution to the Lichnerowicz equation \eqref{eqMainPDE-2}, namely
\[
\Delta_f u+Ae^{2u}+Be^{-2u}+D=0,
\]
where $A, B, D$ are constants satisfying $A\leqslant 0, B\geqslant 0$ and $D\geqslant 0$. If $u$ is bounded, then $u$ is constant.
\end{corollary}

\begin{proof}
Let $u$ solve \eqref{eqMainPDE-2} and suppose that $-C < u < C$ for some $C > 0$. By a scaling argument, the function $v=2u+2C+1$ solves \eqref{eqMainPDE-3}, namely
\[
\Delta_fv+Ae^{-v}+Be^{-v}+D=0
\]
and $1 < v < 4C+1$ everywhere. Here $A, B, D$ are constants which are different from those in Corollary \ref{cor42}, but we still have that $A\leqslant 0, B\geqslant 0, D\geqslant 0$. By Theorem \ref{thmGE-II}, we obtain the following gradient estimate
\[\frac{|\nabla v|}{\sqrt[]{v}}\leqslant c(n)\ \sqrt[]{4C+1}\Big(\frac{1+\sqrt{|\alpha|R}+\sqrt[]{4C+1}}{R}+\frac{1}{\sqrt{t-t_0+T}}+\sqrt{H_2}\Big),
\]
where
\[
H_2=\max\limits_{M\times (0, \infty)}\left\{Ae^v\frac{2v-1}{2v}-Be^{-v}\frac{2v+1}{2v}-\frac{D}{2u}, 0 \right\}=0. 
\]
Here we have used $A\leqslant 0, B\geqslant 0, D\geqslant 0$. Since $v$ does not depend on $t$, first let $t$ tend to infinity then let $R$ approach to infinity to conclude that $|\nabla v| = 0$. Therefore, $v$ is constant, so is $u$. The proof is complete.
\end{proof}

Using the same argument as in the proof of Theorem \ref{thmGEUnderRicci-I}, we can derive gradient estimates for positive solutions of the Einstein-scalar field Lichnerowicz equation under the Ricci flow. Moreover, repeating the proof of Corollary \ref{Harnack}, one can obtain Harnack-type inequality for bounded solution of \eqref{eqMainPDE-2}. We leave them as exercises for interested readers.


\section{Gradient estimates on manifolds with boundary: Proof of Theorems \ref{thmGEWBoundary-I} and \ref{thmGEWBoundary-II}}
\label{sec-GE-bdry}

In this section, we assume that $M$ is an $n$-dimensional complete compact manifold with non-empty boundary $\partial M$. Let $\nu$ be the outward pointing unit normal vector to $\partial M$ and let $II$ stands for the second fundamental form of $\partial M$ with respect to $\nu$. In the paper \cite{LY}, Li--Yau proved that if $M$ is a compact Riemannian manifolds with non-negative Ricci curvature and the boundary $\partial M$ is convex in the sense that $II\geqslant 0$, then any non-negative solution $u$ of the heat equation $\Delta u-\partial_tu=0$ on $M\times(0, +\infty)$ with Neumann boundary condition $\partial_\nu u=0$ satisfies
$$\frac{|\nabla u|^2}{u^2}-\frac{u_t}{u}\leqslant \frac{n}{2t}$$
on $M\times(0, +\infty)$. 

Later, Li-Yau's gradient estimates were generalized to manifolds with non-convex boundary by Chen \cite{Chen} and Wang \cite{Wang}. In \cite{Bai}, the author obtained gradient estimates when the underlying manifold is compact with non-convex boundary evolving under the Ricci flow. We would like to point out that there are some technique complication due to the non-convexity of the boundary since estimates necessarily involve the second fundamental form of $\partial M$ and a so-called interior rolling ball condition that we are going to explain. The interior rolling ball condition is a geometric condition on the boundary $\partial M$ to ensure that the first Neumann eigenvalue is bounded away from zero, see \cite{Chen}, and that the second fundamental form is bounded from above, see \cite{ChenSung}. For clarity and convenience, let us recall its definition.

\begin{definition}\label{def-InteriorRollingRBall}
Let $\partial M$ be the boundary of a compact Riemannian manifold $M$. Then $\partial M$ satisfies the interior rolling $R$-ball condition if for each point $p\in \partial M$ there is a geodesic ball $B_q(R/2)$, centered at some $q\in M$ with radius $R/2$, such that $\{p\}=B_q(R/2)\cap\partial M$ and $B_q(R/2)\subset M$.
\end{definition}


Now, we are going to prove Theorem \ref{thmGEWBoundary-I}. Our proof mainly follows arguments used in \cite{Chen}, in \cite{Bri}, and in \cite{Wang}.

\begin{proof}[Proof of Theorem \ref{thmGEWBoundary-I}]To overcome the non-convexity of the boundary, we need to have the auxiliary cut-off function which was introduced in \cite{Chen, Wang}. Let $\psi: [0,+\infty)\to \mathbb{R}$ be a non-negative $C^2$-function such that $\psi(r)\leqslant H$ if $r\in[0,1/2)$, $\psi(r)=H$ if $r\in[1,+\infty)$, $\psi(0)=0$, $0\leqslant \psi'(r)\leqslant 2 H$, $\psi'(0)=H$, and $\psi''\geqslant -H$. Then, we define
\[
\phi(x)=\psi\Big(\frac{r(x)}{R}\Big),
\]
where $r(x)$ denotes the distance from $x$ to $\partial M$. Finally, we denote
\[
\chi(x)=(1+\phi(x))^2.
\]
As always, we suppose that $0< u\leqslant 1$ is a bounded solution of \eqref{neuman}. We let $h=\log u$ and $w=|\nabla(1-h)|^2$. For each fixed $T<+\infty$, on the compact set $\overline{M}\times[0, T]$, we define
\[
F(x, t)=t\chi (x) w(x) =t\chi (x) \frac{|\nabla h|(x)^2}{(1-h(x))^2}.
\]
Since $F(x,t)$ is a continuous function on the close set $\overline{M}\times [0, T]$, there exists some $(p, t_0)\in \overline{M}\times[0,T]$ such that $F$ achieves its maximum value at $(p, t_0)$. If $F(p, t_0)= 0$, then $F(x, t)\equiv0$ on $\overline{M}\times [0, T]$. Consequently, this implies that the right hand side of \eqref{eqGE-bdry-I} vanishes everywhere in $M\times(0, T]$. Since $T$ is arbitrary, the conclusion in Theorem \ref{thmGEWBoundary-I} follows. Hence, it suffices to consider the case $F(p, t_0)>0$ and $t_0>0$.

In the next stage of the proof, we claim that $p\in\overline{M}\setminus\partial M$. Indeed, by way of contradiction, we suppose that $p\in\partial M$. Then $\partial_\nu F (p,t_0)\geqslant 0$. Let $\{e_1,e_2,\cdots,e_n\}$ be an orthonormal frame at $p$ with a convention that $e_n=\nu$. Clearly,
\[
0\leqslant \frac{\partial F}{\partial \nu}(p,t_0)=t_0\Big( \frac{\partial\chi}{\partial\nu}(p)\frac{|\nabla h|^2}{(1-h)^2}+\chi(p)\frac{2\langle \nabla h,(\nabla h)_\nu \rangle}{(1-h)^2}+\chi(p)\frac{2|\nabla h|^2h_\nu}{(1-h)^3}\Big).
\]
Since $h_\nu=h_n= \partial_\nu u / u=0$ on $\partial M$ and $t_0>0$, on one hand we conclude that 
$$
0\leqslant \frac{\partial\chi}{\partial\nu}(p)\frac{1}{\chi(p)}+\frac{2\langle\nabla h,(\nabla h)_\nu\rangle}{|\nabla h|^2}=\frac{\partial\chi}{\partial\nu}(p)\frac{1}{\chi(p)}+ 2 |\nabla h|^{-2}\sum\limits_{\alpha=1}^{n-1}h_\alpha h_{\alpha\nu}.
$$
Here we used $|\nabla h|(p, t_0)\not=0$ since $F(p, t_0)>0$. However, on the other hand, in terms of the second fundamental form $II=(II_{\alpha\beta})$, we easily obtain
$$h_{\alpha\nu}=-\sum\limits_{\beta=1}^{n-1}II_{\alpha\beta}h_{\beta};$$
see \cite{LY, Chen, Bai}. Therefore, 
$$2\langle\nabla h,(\nabla h)_\nu\rangle=-2II(\nabla h,\nabla h)\leqslant 2H|\nabla h|^2.$$ 
Thus, if we choose $R<1$ then
\[
\frac{\partial\chi}{\partial\nu}(p)\frac{1}{\chi(p)}+\frac{2\langle\nabla h,(\nabla h)_\nu\rangle}{|\nabla h|^2}\leqslant -\frac{2H}{R}+2H<0,
\]
which gives us a contradiction. Thus, $p\in \overline{M}\setminus\partial M$ as claimed. 

Since $F$ obtains its maximum value at $(p, t_0)$, we have the following facts: $\nabla F=0$, $\partial_t F \geqslant 0$, $\Delta F\leqslant 0$ at $(p,t_0)$. Since $u$ satisfies $\eqref{neuman}$ and by Lemma \ref{lemmamain}, the following inequality
$$\Delta w-w_t\geqslant -2C_1w+2(1-h)w^2+\frac{2h}{1-h}\left\langle\nabla w, \nabla h\right\rangle$$
holds on $\overline{M}\times[0, T]$ for any fixed $T>0$. Therefore,
\[
\begin{split}
0\geqslant &\Delta F-F_t=t_0\chi(\Delta w-w_t)+t_0w\Delta\chi+2t_0\langle\nabla\chi,\nabla w\rangle-\chi w\\
\geqslant & t_0\chi \Big( -2C_1w+2(1-h)w^2+\frac{2h}{1-h}\langle \nabla w,\nabla h\rangle\Big) \\
&+t_0w\Delta\chi+2t_0\langle\nabla\chi,\nabla w\rangle-\chi w.
\end{split}
\]
We observe from $0=\nabla F=t_0\chi\nabla w+t_0w\nabla\chi$ that
\[
\langle\nabla\chi,\nabla w\rangle=-\frac{|\nabla\chi|^2}{\chi}w\geqslant -16\frac{H^2}{R^2}w,
\]
and that
\[
\langle \nabla w,\nabla h \rangle=-\Big\langle \frac{\nabla\chi}{\chi},\nabla h\Big\rangle w.
\]
This implies
\begin{equation}\label{qad}
\begin{split}
0\geqslant &t_0 \chi \Big( -2C_1w+2(1-h)w^2-\frac{2h}{1-h}\Big\langle \frac{\nabla\chi}{\chi},\nabla h\Big\rangle w\Big) \\
& +t_0w\Delta\chi-32t_0w\frac{H^2}{R^2}-\chi w.
\end{split}
\end{equation}
Denote
\[
\partial M(R)=\{x\in M|r(x)\leqslant R\}.
\] 
By using an index comparison theorem in \cite{Warner}, see also \cite{Kasue, FYWang, Wang}, we can estimate $\Delta r$ from below as follows 
$$
\Delta r\geqslant -(n-1)(3H+1)
$$
for $x\in \partial M(R)$. Therefore,
$$
\Delta\phi=\frac{1}{R}\psi'\Delta r+\frac{1}{R^2}\psi''|\nabla r|^2\geqslant -\frac{2(n-1)H(3H+1)}{R}-\frac{H}{R^2}.
$$
Now the lower bound of $\Delta \phi$ implies that
$$
\Delta \chi=2(1+\phi)\Delta\phi+2|\nabla\phi|^2\geqslant 2(1+H)\Big(-\frac{2(n-1)H(3H+1)}{R}-\frac{H}{R^2}\Big).
$$
Plugging this inequality into \eqref{qad}, we obtain 
\begin{equation}\label{dkh}
\begin{split}
0\geqslant &-2t_0w
\left\{
\begin{split}
& (1+H)^2C_1 +\frac{2(1+H)(n-1)H(3H+1)}{R}  \\
&+\frac{16H^2+(1+H)H}{R^2}+\frac{\chi}{2t_0}
\end{split}
\right\}\\
&+2t_0(1-h)w^2-\frac{2ht_0}{1-h}\left\langle \nabla\chi,\nabla h\right\rangle w \\
=&-2\Big(E+\frac{\chi}{2t_0}\Big)t_0w
+2t_0(1-h)w^2-\frac{2ht_0}{1-h}\left\langle \nabla\chi,\nabla h\right\rangle w,
\end{split}
\end{equation}
where
\[
E:=(1+H)^2C_1 + 2(1+H)(n-1)H(3H+1)/R+\big(16H^2+(1+H)H \big) R^{-2}.
\]
Now, we want to estimate both the first term and the third term in the right hand side of \eqref{dkh}. To control the first term, we use the Cauchy--Schwarz inequality to obtain
$$
-2t_0w\Big(E+\frac{\chi}{2t_0}\Big)\geqslant -\frac{1}{2}t_0w^2-\Big(4t_0E^2+\frac{(1+H)^4}{t_0}\Big).
$$
By the Young inequality, we estimate the third term by
\begin{align*}
\frac{2ht_0}{1-h}\left\langle\nabla\chi,\nabla h\right\rangle w\leqslant &2t_0|h||\nabla\chi|w^{3/2}  = 2\big[t_0(1-h)w^2\big]^{3/4}\frac{t_0^{1/4}|h||\nabla\chi|}{\big(1-h\big)^{3/4}} \\
\leqslant &t_0(1-h)w^2+\frac{27}{16}\frac{t_0(h|\nabla\chi|)^4}{\big(1-h\big)^3} \\
\leqslant &t_0(1-h)w^2+432\frac{h^4}{(1-h)^3}\frac{t_0(1+H)^4H^4}{R^4}.
\end{align*}
Plugging the above two inequalities into \eqref{dkh}, we know that 
$$
t_0(1-h)w^2 \leqslant \frac{1}{2}t_0w^2+4t_0E^2+\frac{(1+H)^4}{t_0}+432\frac{h^4}{(1-h)^3}\frac{t_0(1+H)^4H^4}{R^4}
$$
at $(p, t_0)$. Since $h/(1-h) \leqslant 1$ and $1/(1-h) \leqslant 1$, we have
$$
t_0w^2\leqslant \frac{1}{2}t_0w^2+4t_0E^2
+\frac{(1+H)^4}{t_0}+432t_0\frac{(1+H)^4H^4}{R^4}.
$$
Consequently, we can estimate $w^2$ from above as follows
\begin{align*}
w^2\leqslant &24(1+H)^4C_1^2+\frac{24[2(1+H)(n-1)H(3H+1)]^2}{R^2}\\
&2\frac{(1+H)^4}{t_0^2}+\frac{24[16H^2+(1+H)H]^2+864(1+H)^4H^4}{R^4}\\
=&24(1+H)^4C_1^2+2\frac{(1+H)^4}{t_0^2}+\frac{C_2^2}{R^2}+\frac{C_3}{R^4}.
\end{align*}
Therefore, for any $x\in \overline{M}$, we get
\begin{align*}
Tw(x, T)&\leqslant T(1+\phi(x))^2w(x, T)\\
&\leqslant t_0(1+\phi(p))^2w(p, t_0)\\
&\leqslant (1+H)^2\Big[\sqrt{24}(1+H)^2t_0C_1+\sqrt{2}(1+H)^2+t_0\Big(\frac{C_2}{R}+\frac{\sqrt{C_3}}{R^2}\Big)\Big]\\
&\leqslant (1+H)^2\Big[\sqrt{24}(1+H)^2TC_1+\sqrt{2}(1+H)^2+T\Big(\frac{C_2}{R}+\frac{\sqrt{C_3}}{R^2}\Big)\Big].
\end{align*}
From this we deduce that 
\[
w(x, T)\leqslant (1+H)^2\left (\sqrt{24}(1+H)^2C_1+\sqrt{2}\frac{(1+H)^2}{T}+\frac{C_2}{R}+\frac{\sqrt{C_3}}{R^2}\right ).
\]
Since $T$ is arbitrary, the proof is complete.
\end{proof}

\begin{remark}\label{rem}
In our estimate above, the condition that $R$ is ``small" is understood in the following sense: $R$ is chosen to be a positive constant less than $1$. Moreover, $R$ is dependent on the upper bound of the sectional curvature of the manifold near the boundary. The upper bound of $R$ is explicitly determined by
$$
\sqrt{K_R}\tan{\big (R\sqrt{K_R}\big )}\leqslant \frac{H}{2}+\frac{1}{2}
$$
and
$$
\frac{H}{\sqrt{K_R}}\tan{R\sqrt{K_R}}\leqslant \frac{1}{2},
$$
where $K_R$ is the upper bound of the sectional curvature on the set $\partial M(R)$; see \cite{Chen}.
\end{remark}

Let us now consider one special case when $\partial M$ is convex, namely, $H \equiv 0$ identically. Using Theorem \ref{thmGEWBoundary-I}, we obtain the following corollary. 

\begin{corollary}\label{main5}
Let $M$ be a compact Riemannian manifold with convex boundary $\partial M$. Suppose that $\partial M$ satisfies the interior rolling $R$-ball condition. Let $K$ be non-negative constant such that the Ricci curvature $\Ric_M$ of $M$ is bounded from below by $-K$. By choosing $R$ ``small", any positive solution $u(x,t)$ of the following equation
\begin{equation} \label{neumann}
\left\{
\begin{split}
u_t &=\Delta u+au\log u+bu+Au^p+Bu^q,\\
\frac{\partial u}{\partial \nu}\Big|_{\partial M} &=0,
\end{split}
\right.
\end{equation}
where $a+b\leqslant 0, A\leqslant 0, B\geqslant 0, p\geqslant 1, q\geqslant 0$ are constants, if satisfies $u \leqslant 1$, then 
\begin{equation} \label{neu}
\frac{|\nabla u|}{u}\leqslant \big(\sqrt[4]{24}\sqrt{K}+ \frac{\sqrt[4]{2}}{\sqrt{t}}\big)(1-\log u)
\end{equation}
on $M\times(0, +\infty)$. 
\end{corollary}

It is worth mentioning that by using Corollary \ref{main5} we can obtain Liouville type theorems for Schr\"{o}dinger-type equations, Yamabe-type equations, as well as Lichnerowicz-type equations on compact manifolds with boundary as in Section \ref{sec-GE-I}. Due to the limit of length, we do not mention further and leave the details for interested readers. 

Note that if we use Lemma \ref{lemmaeqMainPDE-3} and repeat the argument in the proof of Theorem \ref{thmGEWBoundary-I}, then it easy to prove Theorem \ref{thmGEWBoundary-II}. Since the proof of Theorem \ref{thmGEWBoundary-II} is similar to that of Theorem \ref{thmGEWBoundary-I}, we omit the details and leave it to the reader.


\section*{Acknowlegement} 

This research is funded by the VNU University of Science under project number TN.16.01.

\end{document}